\documentclass{ws-ijac}

\usepackage{xcolor}
\usepackage[verbose]{hyperref}
\hypersetup{colorlinks=false,allbordercolors=blue,pdfborderstyle={/S/U/W 1}}

\usepackage[T1]{fontenc}
\usepackage{lmodern}
\usepackage{amssymb}
\usepackage[all]{xy}
\usepackage{tikz}
\usepackage{graphicx}
\usepackage{url}
\usepackage{tikz-cd}
\usepackage{pgf}
\usepackage{ragged2e}
\usepackage{hyperref}
\usepackage{tikz}
\usetikzlibrary{arrows,decorations.pathmorphing,backgrounds,positioning,fit,petri,calc}

\usepackage[utf8]{inputenc} 
\usepackage[T1]{fontenc}    
\usepackage{booktabs}       
\usepackage{nicefrac}       
\usepackage{microtype}      
\usepackage{lipsum}

\newtheorem{lema}{Lemma}[section]
\newtheorem{teo}[lema]{Theorem}
\newtheorem{cor}[lema]{Corollary}
\newtheorem{prop}[lema]{Proposition}

\newtheorem{defi}{Definition}

\numberwithin{equation}{section}

\newcommand{\tq}{\,|\,}

\newcommand{\ed}{\ar@{-}}

\renewcommand{\epsilon}{\varepsilon}

\numberwithin{equation}{section}

\usepackage{graphicx}

\begin{document}

\markboth{Kisnney Emiliano de Almeida and Igor Lima}
{Subgroup Separability of Artin groups II}

%
\catchline{}{}{}{}{}
%

\title{Subgroup Separability of Artin groups II
}

\author{Kisnney Emiliano de Almeida}

\address{Departamento de Ciências Exatas, Universidade Estadual de Feira de Santana (UEFS),\\ Av. Transnordestina S/N, 44036-900,  Feira de Santana - BA,  Brazil\\
\textit{kisnney@gmail.com}}

\author{Igor Lima}

\address{Departamento de Matem\'atica, Universidade de Bras\'ilia,\\
Brasilia-DF, 70910-900, Brazil\\
\textit{igor.matematico@gmail.com}}

\maketitle

\begin{history}
\received{(Day Month Year)}
\revised{(Day Month Year)}
\accepted{(Day Month Year)}
\published{(Day Month Year)}
\comby{}
\end{history}

\begin{abstract}
In this short note, we establish a more concrete characterization of subgroup separable (LERF) Artin groups, building on a previously known criterion. Specifically, we prove an Artin group is LERF if and only if its underlying graph does not contain certain induced subgraphs, thus providing a more direct generalization of the Metaftsis-Raptis criterion for RAAGs. Consequently, we prove every LERF Artin group is coherent and obtain a graph condition for coherent Artin groups to be LERF.
\end{abstract}



\section{Introduction}

We commence by recalling that the cosets of finite index normal subgroups of a group $G$ form a basis for the {\bfseries profinite topology} of $G$. A group $G$ is called {\bfseries Residually Finite} if this topology is Hausdorff, or equivalently, if the trivial subgroup is closed in it. That is equivalent to the intersection of all its finite index subgroups being trivial, among other possible definitions. There are some natural stronger versions of this property. We say a group $G$ is {\bfseries Locally Extended  Residually Finite}, or just {\bfseries LERF}, if every finitely generated subgroup of $G$ is closed in the profinite topology of $G$. A LERF group is also called {\bfseries subgroup separable}, which is equivalent to every finitely generated subgroup of $G$ to be equal to an intersection of finite index subgroups of $G$. Subgroup separability was originally introduced by Hall \cite{H}, possessing significant implications for both group theory and low-dimensional topology. Nevertheless, it has been affirmatively or negatively established for only a limited number of group classes.

Artin groups represent a diverse class of groups, defined by a finite simplicial labeled graph. They encompass a wide range of groups, including free groups, free abelian groups, classical Braid groups, and various other subclasses. Many conjectures surround Artin groups, with most of them lacking definitive answers. For instance, there are ongoing conjectures establishing torsion-freeness, solvable word problem, centerlessness (for irreducible non-spherical Artin groups), among other unsolved questions. Further details can be found at \cite{Ca}. 

Metaftsis and Raptis \cite[Theorem 2]{MR} established that a right-angled Artin group is subgroup separable if and only if its underlying graph does not contain certain induced subgraphs. In a previous work   \cite{AL} we have combined the Metaftsis-Raptis criterion and Jankiewicz-Schreve results on the Generalized Tits Conjecture \cite{JS} to establish a general criterion for subgroup separability of all Artin groups. However, this general criterion has a different, more constructive shape. We will give more details on Artin groups and previous results on their subgroup separability in Section \ref{preliminaries}.

In this paper, we give a new criterion for subgroup separability of Artin groups that directly generalizes the Metaftsis-Raptis criterion for RAAGs. Throughout this paper, all graphs are finite simplicial and with edges labeled by integers greater than 1. We recall a subgraph $\Theta$ of a graph $\Gamma$ is called {\bfseries induced}  if every edge of $\Gamma$ which connects vertices of $\Theta$ is also inside $\Theta$.  To simplify notations and avoid ambiguity, we will use the term {\bfseries path} to refer exclusively to a non-closed path.

\begin{defi}\label{defpoisgr}
  We say a graph $\Gamma$ is \emph{poisoned} if it contains an induced subgraph $\gamma$ that satisfies one of the following conditions:
  \begin{enumerate}
  \item
  $\gamma$ is a square of edges labeled by 2;

  \item
  $\gamma$ is a path of length three with all edges labeled by 2;

  \item
  $\gamma$ is connected with exactly three vertices and at most one edge labeled by 2;

  \item
  $\gamma$ is  isomorphic to the graph
 $$ \begin{tikzpicture}
  \node[circle,draw](a) {};
\node[circle,draw] (b) [right=of a] {};
  \node[circle,draw] (c) [below=of b] {};
  \node[circle,draw] (d) [below=of a] {};

  \draw[-] (a) -- node[label=above:$2$]{} (b);
  \draw[-] (b) -- node[label=right:$2$]{} (c);
    \draw[-] (c) -- node[label=below:$2$]{} (d);
      \draw[-] (d) -- node[label=left:$2$]{} (a);
        \draw[-] (a) -- node[label=above:$m$]{} (c);
      \end{tikzpicture},$$

    with $m>2$.

  \item
  $\gamma$ is  isomorphic to the graph

   $$ \begin{tikzpicture}
  \node[circle,draw](a) {};
\node[circle,draw] (b) [right=of a] {};
  \node[circle,draw] (c) [below=of b] {};
  \node[circle,draw] (d) [below=of a] {};

  \draw[-] (a) -- node[label=above:$2$]{} (b);
  \draw[-] (b) -- node[label=right:$2$]{} (c);
    \draw[-] (c) -- node[label=below:$2$]{} (d);
      \draw[-] (d) -- node[label=left:$2$]{} (a);
        \draw[-] (a) -- node[label=above:$m$,pos=0.4]{} (c);
        \draw[-] (b) -- node[label=below:$n$,pos=0.6]{} (d);
      \end{tikzpicture},$$

      with $m>2$ and $n>2$.
    \end{enumerate}

    For clarity, we will say a graph is \emph{$i$-poisoned} if it contains an induced subgraph $\gamma$ that satisfies the $i$-th condition above.
\end{defi}

Note that if $\Gamma_1$ is a poisoned induced subgraph of $\Gamma_2$ then $\Gamma_2$ is poisoned. We present our main result below.

\begin{teo}\label{teopois}
  Let $A=A(\Gamma)$ be an Artin group. Then $A$ is subgroup separable if and only if $\Gamma$ is not poisoned.
\end{teo}

We recall a group $G$ is called {\bfseries coherent} if every finitely generated subgroup of $G$ is finitely presented. Obviously coherence carries over to subgroups and it is well known that finite groups, free groups and fundamental groups of $2$ (or $3$)-manifolds, for instance, are coherent. A good survey on the subject may be found at \cite{Wi3}. The full description of coherent Artin groups is already established, as we will explain in Section \ref{app}. By using Theorem \ref{teopois} we can relate the two properties.

\begin{cor}\label{teocoh}
  Every subgroup separable Artin group is coherent. A coherent Artin group $A=A(\Gamma)$ is subgroup separable if and only if every induced path of $\Gamma$ with length greater than 1 has length 2 with both edges labeled by 2.   
\end{cor}
 
This paper is organized as follows: in Section 2 we explain some basic definitions and previously known results on the subject; in Section 3 we prove the main result; and in Section 4 we prove Corollary \ref{teocoh}.

\section{Preliminaries}\label{preliminaries}

Let $\Gamma$ be a finite simplicial graph, with edges labeled by integers greater than 1. Then the {\bfseries Artin group with  $\Gamma$ as underlying graph}, denoted by $A(\Gamma)$, is given by a finite presentation, with generators corresponding to the vertices of $\Gamma$ and relations given by  \[  \underbrace{abab\cdots}_{m \text{ factors}}=\underbrace{baba\cdots}_{m \text{ factors}} \] for each edge of $\Gamma$ that connects the vertices $a$ and $b$ labeled by $m$. In that case we will say  $a$ and $b$ are {\bfseries $m$-adjacent}.

Let $V(\Gamma)$ be the set of vertices of $\Gamma$ and $E(\Gamma)$ be its set of edges. A subgraph $\gamma$ of $\Gamma$ is called an {\bfseries induced} subgraph (generated by $V(\gamma)$) if every two vertices of $\gamma$ that are adjacent in $\Gamma$ are also adjacent in $\gamma$. We will denote the induced subgraph generated by a subset $V_1$ of $V(\Gamma)$ by $S(V_1)$. A subset $V_0\subset V(\Gamma)$ generates a subgroup of $A(\Gamma)$ that is isomorphic to $A(S(V_0))$ \cite{L} - such subgroups are called {\bfseries parabolic subgroups} of $A(\Gamma)$.

We will strongly use the behaviour of subgroup separability under taking subgroups.

\begin{teo}[\cite{S},\cite{S2}]\label{teosublerf}
Let $H$ be a subgroup of $G$. If $H$ is not LERF then $G$ is not LERF either. The converse holds if \ $[G:H] < \infty$.
\end{teo}

We define a {\bfseries 2-cone} of a graph $\Gamma'$ (with apex $u$), denoted by $\Gamma=\Gamma'+\{u\}$, as the graph $\Gamma$ with $V(\Gamma)=V(\Gamma')\sqcup \{u\}$ and such that $u$ is 2-adjacent in $\Gamma$ to every vertex of $\Gamma'$. Let $\mathcal{S}$ be the class of graphs  obtained from graphs with one or two vertices by taking disjoint unions and 2-cones.  The result  below is the known characterization of subgroup separable Artin groups.

\begin{teo}\label{teogeral}(\cite{AL})
  Let $A=A(\Gamma)$ be an Artin group, with $\Gamma$ as its underlying graph. Then $A$ is subgroup separable if and only if $\Gamma\in \mathcal{S}$.
\end{teo}

If $\Gamma$ is a graph, we also define
$$Z(\Gamma):=\{u\in V(\Gamma)\tq \text{$u$ is 2-adjacent to every $v\in V(\Gamma)$ such that $v\neq u$}\}$$

and recall the lemma below, which is essentially a particular case of Theorem \ref{teogeral}.

\begin{lema} \label{lemacongamma}(\cite{AL}, Lemma 4.6)
 Let $A=A(\Gamma)$ be a LERF Artin group such that $\Gamma$ is connected and has at least three vertices. Then $Z(\Gamma)\neq \emptyset$.
\end{lema}

\section{Proof of Theorem \ref{teopois}}

 Suppose $\Gamma$ is poisoned. Then $\Gamma$ contains an induced subgraph $\gamma$ that satisfies one of the conditions of Definition \ref{defpoisgr}. Then $A(\gamma)$ is not LERF by Lemma \ref{lemacongamma} hence $A$ is not LERF by Theorem \ref{teosublerf}.

  Now suppose $\Gamma$ is not poisoned and let $V=V(\Gamma)$. We will proceed by induction on $|V|$ to prove that $A$ is LERF. By Theorem \ref{teogeral} it is enough to prove $\Gamma \in \mathcal{S}$.

  If $|V|$ is 1 or 2 then $\Gamma\in \mathcal{S}$ by definition. If $|V|=3$, then either $\Gamma$ is disconnected or $\Gamma$ is a 2-cone - either way, $\Gamma\in \mathcal{S}$.

   Let $n>3$, suppose the result is true for $|V|<n$ and let $A(\Gamma)$ be an Artin group such that $|V|=n$.  
   
   We claim every non-empty proper induced subgraph of $\Gamma$ belongs to $\mathcal{S}$. Indeed, let $\Gamma'$ be a non-empty proper induced subgraph of $\Gamma$. Since $\Gamma$ is not poisoned, neither is $\Gamma'$. So, by the induction hypothesis, $A(\Gamma')$ is LERF, hence $\Gamma'\in \mathcal{S}$ by Theorem \ref{teogeral}. We will freely use this fact throughout the proof. 
   
   Note that it is enough to prove  $Z(\Gamma)\neq \emptyset$, since in this case $\Gamma$ decomposes as a 2-cone of a proper induced subgraph, which belongs to $\mathcal{S}$ by the claim above.

 Suppose $\Gamma$ is disconnected. Then $\Gamma=\Gamma_1\cup \Gamma_2$ such that $\Gamma_1$ and $\Gamma_2$ are non-empty proper induced subgraphs of $\Gamma$. Then $\Gamma_1,\Gamma_2\in \mathcal{S}$ so $\Gamma\in \mathcal{S}$.
 
 So we may assume $\Gamma$ is connected.

  Let $v\in V$ and let $\Gamma_0=S(V\setminus \{v\})$. Note that $\Gamma_0\in \mathcal{S}$, since it is a non-empty proper induced subgraph of $\Gamma$.

  Now suppose $\Gamma_0$ is disconnected, let's say $\Gamma_0=\Gamma_{01}\cup \cdots \cup \Gamma_{0k}$, with $\Gamma_{0i}$ connected for all $i$ and $k\geq 2$. Since $\Gamma$ is connected, then $v$ is adjacent to $\Gamma_{0i}$ for all $i$, let's say the edge $e_i$ connects $v$ to a vertex $u_{0i}$ of $\Gamma_{0i}$. Then the edges $e_i$ are all labeled by 2, otherwise $\Gamma$ would be $3$-poisoned. If there was a vertex of $\Gamma_{0i}$ which is $m$-adjacent to $v$, with $m>2$, then $\Gamma$ would also be $3$-poisoned, so every vertex of $\Gamma_0$ that is adjacent to $v$ is actually 2-adjacent to $v$. If all the vertices of $\Gamma_0$ are 2-adjacent to $v$ then $v\in Z(\Gamma)$ hence $\Gamma\in \mathcal{S}$. That means we may assume there is a vertex $w\in \Gamma_{0i}$ for some $i$ such that $w$ is not adjacent to $v$. Assume, without loss of generality, that $w\in \Gamma_{01}$. Note that $w\neq u_{01}$. 

  Suppose $w$ is adjacent to $u_{01}$. Then $w$ is actually 2-adjacent to $u_{01}$ otherwise $\Gamma$ would be 3-poisoned with $\gamma=S(w,u_{01},v)$. However that would mean $\Gamma$ is 2-poisoned, with $\gamma=S(w,u_{01},v,u_{02})$. That implies $w$ is not adjacent to $u_{01}$. Since $\Gamma_{01}$ is connected, that means $\Gamma_{01}$ has at least three vertices. On the other hand, $A(\Gamma_{01})$ is LERF since $\Gamma_{01}$  is a non-empty proper induced subgraph of $\Gamma$. By Lemma \ref{lemacongamma} there is a vertex $u_{01}'\in Z(\Gamma_{01})$. Since $u_{01}'$ is 2-adjacent to $w$ which is not adjacent to $u_{01}$, we have $u_{01}\neq u_{01}' \neq w$. Then $u_{01}'$ is adjacent (hence 2-adjacent) to $v$, otherwise $\Gamma$ would be 2-poisoned with $\gamma=S(u_{01}',u_{01},v,u_{02})$. But that would also imply $\Gamma$ is 2-poisoned, with $\gamma=S(w,u_{01}',v,u_{02})$, which is a contradiction.

 So we may also assume $\Gamma_0$ is connected. Since $\Gamma_0$ is a proper induced subgraph of $\Gamma$ with $|V|-1\geq 3$ vertices, by Lemma \ref{lemacongamma} there is a vertex $u\in Z(\Gamma_0)$, such that $\Gamma_0=\Gamma_1+\{u\}$. Let $V(\Gamma_1)=\{w_1,w_2,\dots, w_k\}$, with $k=|V|-2\geq 2$. 

If $u$ and $v$ are 2-adjacent, then $u\in Z(\Gamma)$ hence $\Gamma\in \mathcal{S}$. So we may assume that $u$ and $v$ are not 2-adjacent.

Suppose $u$ and $v$ are $m$-adjacent for $m>2$ and let $i\in \{1,\dots, k\}$. Since $S(u,v,w_i)$ is not 3-poisoned, then $w_i$ and $v$ are 2-adjacent. Now, let $j\in \{2,\dots, k\}$. Since $S(u,v,w_1,w_j)$ is neither 4-poisoned nor 5-poisoned then $w_1$ and $w_j$ are 2-adjacent. That means $w_1\in Z(\Gamma)$ and hence $\Gamma\in \mathcal{S}$.

Finally, suppose that $u$ and $v$ are not adjacent. If $i\in \{1,\dots,k\}$ and $w_i$ is $m$-adjacent to $v$ for $m>2$ then $S(u,v,w_i)$ is 3-poisoned, which is a contradiction. So each $w_i$ is either 2-adjacent or not adjacent to $v$. Since $\Gamma$ is connected, there is some $i_0\in \{1,\dots, k\}$ such that $w_{i_0}$ is 2-adjacent to $v$ - without loss of generality, we may assume $i_0=1$. 

Let $j\in \{2,\dots,k\}$. If $w_j$ is $m$-adjacent to $w_1$ for $m>2$, since $S(v,w_1,w_j)$ is not 3-poisoned then $w_j$ is 2-adjacent to $v$, which means $S(u,v,w_1,w_j)$ is 4-poisoned, a contradiction. If $w_j$ is not adjacent to $w_1$, then either $S(u,v,w_1,w_j)$ is 1-poisoned (if $v$ is 2-adjacent to $w_j$) or 2-poisoned (if $v$ is not adjacent to $w_j$), which is also a contradiction. Then $w_j$ is 2-adjacent to $w_1$, which implies $w_1\in Z(\Gamma)$ hence $\Gamma\in \mathcal{S}$, which concludes the proof.

\section{Subgroup separability and coherence}\label{app}

We recall a graph is called \emph{chordal} if it contains no cycle of length greater than or equal to 4 as an induced subgraph. The result below was rewritten under our notation.

\begin{teo}[\cite{G,Wi2}]\label{teoartincoh}
An Artin group $A=A(\Gamma)$ is coherent if and only if
$\Gamma$ is chordal, every complete subgraph of $\Gamma$ with 3 or 4 vertices has at most one edge label $>2$, and $\Gamma$ is not 4-poisoned.
\end{teo}

We will split the proof of Corollary \ref{teocoh} into the two next results.

\begin{prop}
Every subgroup separable Artin group is coherent.
\end{prop}

\begin{proof}
  Let $A=A(\Gamma)$ be a subgroup separable Artin group.  By Theorem \ref{teopois}, $\Gamma$ is not poisoned. Let $\gamma$ be an induced subgraph of $\Gamma$.

  First we prove $\Gamma$ is chordal. Suppose $\gamma$ is a cycle of length $l$ greater or equal to 4. If all edges of $\gamma$ are labeled by 2 then either $\Gamma$ is 1-poisoned (if $l=4$) or 2-poisoned (if $l>4$), which is a contradiction. Then $\gamma$ contains an edge labeled by $m>2$, generated by vertices $u,v$. Since $\gamma$ is connected, then $v$ is adjacent to another vertex $w$ of $\gamma$. Since $l\geq 4$, then $S(u,v,w)$ is 3-poisoned and so is $\Gamma$, which is a contradiction. Then there can be no such $\gamma$ hence $\Gamma$ is chordal.

Now suppose $\gamma$ is complete.

If $\gamma$ has three vertices, since $\Gamma$ is not 3-poisoned then $\gamma$ has at most one edge label $>2$.

If $\gamma$ has four vertices, suppose it contains two edges labeled by $m>2$ and $n>2$. Those edges may not have a vertex in common, otherwise $\Gamma$ would be 3-poisoned. For the same reason, all other edges of $\gamma$ need to be labeled by 2. That means $\Gamma$ is 5-poisoned, which is a contradiction.

Since $\Gamma$ is obviously not 4-poisoned, then $A$ is coherent by Theorem \ref{teoartincoh}.
\end{proof}

\begin{prop}
Let $A=A(\Gamma)$ be a coherent Artin group. Then $A$ is subgroup separable if and only if every induced path of $\Gamma$ which has length greater than 1  has length 2 with both edges labeled by 2.
\end{prop}

\begin{proof}
  Let $A=A(\Gamma)$ be a coherent Artin group.

   First suppose every induced path of $\Gamma$ with length greater than 1  has length 2 with both edges labeled by 2. By Theorem \ref{teopois} it is enough to prove $\Gamma$ is not poisoned. The condition above guarantees that $\Gamma$ is not 2-poisoned. By Theorem \ref{teoartincoh}, $\Gamma$ is chordal hence not 1-poisoned. The last two conditions of the same theorem also tell us $\Gamma$ is not 4-poisoned nor 5-poisoned.

Suppose $\Gamma$ is 3-poisoned. Then there is an induced subgraph $\gamma$ of $\Gamma$ which has exactly three vertices and at most one edge labeled by 2. If $\gamma$ is complete, that contradicts Theorem \ref{teoartincoh}; if it is not complete then that contradicts our hypothesis. Hence $\Gamma$ is not poisoned.

  Now suppose $A$ is subgroup separable. By Theorem \ref{teopois} $\Gamma$ is not poisoned. Let $\gamma$ be an induced path of $\Gamma$ of length greater than 1. First suppose $\gamma$ contains an edge labeled by $m>2$ generated by vertices $u,v$. Since $\gamma$ is not an edge then there is a third vertex $w$ which is adjacent to $u$ or $v$. That means $S(u,v,w)$ is 3-poisoned and so is $\Gamma$, which is a contradiction. Then $\gamma$ has all edges labeled by 2. That implies the length of $\gamma$ is equal to 2, otherwise $\Gamma$ would be 2-poisoned, which concludes the proof.
\end{proof}

{\bfseries Acknowledgments}: We thank A. Krasilnikov and P. Zalesskii for the insightful questions that inspired this work. We also thank the anonymous referee for the careful reading, valuable suggestions, and insightful remarks. The first author was partially supported by FINAPESQ-UEFS, FAPDF, Brazil. The second author was partially supported by DPI/UnB, FAPDF, Brazil.

\section*{ORCID}

\noindent Kisnney Emiliano de Almeida - \url{https://orcid.org/0000-0002-9713-2573}

\noindent Igor Lima - \url{https://orcid.org/0000-0002-0346-2716}

\end{document}